\documentclass[a4paper,10pt]{article}
\usepackage{QED}
\usepackage{amsmath}
\usepackage{amssymb}
\usepackage{a 4}

\def\G{\Gamma}

\def\l{\lambda}

\def\m{\mu}

\def\f{\rightarrow}

\def\<{\langle}
\def\>{\rangle}
\def\F{\displaystyle\frac}
\newtheorem{theorem}{Theorem}[section]
\newtheorem{lemma}{Lemma}[section]
\newtheorem{remark}{Remark}[section]

\newtheorem{definition}{Definition}[section]
\newtheorem{notation}{Notation}[section]

\begin{document}
\vspace*{0cm}
\begin{center}

{\Large \bf A semantical proof of the strong normalization theorem for 
full propositional classical natural deduction}\\ [1cm]

{\bf  Karim NOUR and Khelifa SABER} \\ 
 LAMA - Equipe de logique \\
  Universit\'e de Chamb\'ery\\
 73376 Le Bourget du Lac\\
e-mail : $\{$knour,ksabe$\}$@univ-savoie.fr\\[0.5cm]

\end{center}

{\small {\bf Abstract} {\em We give in this paper a short semantical 
proof of the strong
  normalization for full propositional classical natural deduction. 
This
  proof is an adaptation of reducibility candidates introduced by
 J.-Y. Girard and simplified to the classical case by M. Parigot.  }}

\section{Introduction}

This paper gives a semantical proof of the strong normalization of the
cut-elimination procedure for full propositional classical logic
written in natural deduction style.  By full we mean that all the 
logical
connectives ($\perp$, $\f$, $\wedge$ and $\vee$) are considered as
primitive. We also consider the three reduction relations (logical, 
commutative
and classical reductions) necessary to obtain the subformula  property
(see \cite{deG2}).

Until very recently (see the introduction of \cite{deG2} for a
brief history), no proof of the strong normalization of the
cut-elimination procedure was known for full logic.

In \cite{deG2}, Ph. De Groote gives such a proof by using a CPS-style
transformation from full classical logic to implicative intuitionistic
logic, i.e., the simply typed $\l$-calculus.  

A very elegant and direct proof of the strong normalization of the
full logic is given in
\cite{JoMa} but only the intuitionistic case is given.

R. David and the first author give in \cite{dav2} a direct and
syntactical proof of this result. This proof is based on a
characterization of the strongly normalizable deductions and a
substitution lemma which stipulates the fact that the deduction
obtained while replacing in a strongly normalizable deduction an
hypothesis by another strongly normalizable deduction is also strongly
normalizable. The same idea is used in \cite{davnou} to give a
short proof of the strong normalization of the simply typed
$\l\m$-calculus of \cite{Par1}.

R. Matthes recently found another semantical proof of this result (see 
\cite{matthes}). His proof
uses a complicated concept of saturated subsets of terms. 

Our proof is a generalization of M. Parigot's strong normalization
result of the $\lambda\mu$-calculus (see \cite{Par2}) for the types of
J.-Y. Girard's system ${\cal F}$ using reducibility candidates. We
also use a very technical lemma proved in \cite{dav2} concerning
commutative reductions. To the best of our knowledge, this is the
shortest proof of a such result. 

The paper is organized as follows. In section 2, we give the syntax of
the terms and the reduction rules. In section 3, we define the
reducibility candidates and establish some important properties. In
section 4, we show an ``adequation lemma'' which allows to prove the
strong normalization of all typed terms.

\section{The typed system}

We use notations inspired by the paper \cite{and}. 

\begin{definition}
\begin{enumerate}
\item The types are built from propositional variables and the constant symbol $\perp$ with the
connectors $\f$, $\wedge$ and $\vee$.

\item Let $\mathcal{X}$ and $\mathcal{A}$ be two disjoint alphabets
for distinguishing the $\lambda$-variables and $\mu$-variables
respectively.  We code deductions by using a set of terms
$\mathcal{T}$ which extends the $\l$-terms and is given by the
following grammars:
\begin{center}
$
\mathcal{T} \; := \;\mathcal{X} \; | \; 
\lambda\mathcal{X}.\mathcal{T}\;
|\; (\mathcal{T}\;\;\mathcal{E}) \; | \; \< \mathcal{T},\mathcal{T} \> 
\; | \;$$
\omega_1 \mathcal{T}$$ \; |\;$$ \omega_2 \mathcal{T}$$ \; | \;
\mu\mathcal{A}.\mathcal{T} \; | \; (\mathcal{A}\; \; \mathcal{T})
$

$
\mathcal{E} \; := \; \mathcal{T} \; | \; $$\pi_1$$ \; | \;$ $\pi_2 $$\; 
|
\; [\mathcal{X}.\mathcal{T},\mathcal{X}.\mathcal{T}]
$
\end{center}

An element of the set $\mathcal{E}$ is said to be an 
$\mathcal{E}$-term.

\item  The meaning of the new constructors is given by the typing rules
below where $\G$ (resp. $\Delta$) is a context, i.e. a set of
declarations of the form $x : A$ (resp. $a : A$) where $x$ is a
$\l$-variable (resp. $a$ is a $\m$-variable) and $A$ is a type.
 
\begin{center}
 $\F{}{\Gamma, x:A\,\, \vdash x:A\,\, ; \, \Delta}{ax}$
\end{center}

\begin{center}
$\F{\Gamma, x:A \vdash t:B;\Delta}{\Gamma \vdash \lambda x.t:A \to 
B;\Delta}{\to_i}
\quad\quad\quad 
\F{\Gamma \vdash u:A \to B;\Delta \quad \Gamma \vdash
v:A;\Delta}{\Gamma\vdash (u\; \; v):B;\Delta}{\to_e}$
\end{center}

\begin{center}
$\F{\Gamma \vdash u:A;\Delta \quad \Gamma \vdash v:B ; \Delta}{\Gamma
\vdash \<u,v\>:A \wedge B ; \Delta}{\wedge_i}$
\end{center}

\begin{center}
$\F{\Gamma \vdash t:A \wedge B ; \Delta}{\Gamma \vdash (t\;\;\pi_1):A ;
\Delta}{\wedge^1_e} 
\quad 
\F{\Gamma \vdash t:A\wedge B ; \Delta}{\Gamma \vdash (t\;\;\pi_2):B ; 
\Delta}{\wedge^2_e}$
\end{center}

\begin{center}
$\F{\Gamma\vdash t:A;\Delta}{\Gamma\vdash \omega_1
    t:A \vee B ;\Delta}{\vee^1_i}
\quad
\F{\Gamma \vdash t:B; \Delta}{\Gamma\vdash \omega_2 t:A\vee B 
;\Delta}{\vee^2_i}$
\end{center}

\begin{center}
$\F{\Gamma \vdash t:A\vee B ;\Delta\quad\Gamma, x:A \vdash u:C ;
\Delta\quad\Gamma, y:B \vdash v:C ; \Delta}{\Gamma \vdash (t\;\;[x.u, 
y.v]):C ; \Delta}{\vee_e}$ 
\end{center}

\begin{center}
$\F{\Gamma\vdash t:A ;\Delta, a:A}{\Gamma \vdash (a\;\;t):\bot ;\Delta,
a:A}{abs_i}
\quad
\F{\Gamma\vdash t:\bot; \Delta, a:A}{\Gamma \vdash\mu a.t:A; 
\Delta}{abs_e}$
\end{center}

\item The cut-elimination procedure corresponds to the  reduction rules 
given bellow. There are three kinds of cuts:

\begin{enumerate}

\item The logical cuts: They appear when the introduction of a 
connective is immediately followed by its elimination. The corresponding rules 
are:

\begin{itemize}
\item $(\lambda x.u \;\; v) \triangleright u[x:=v]$

\item $(\<t_1,t_2\>\;\;\pi_i) \triangleright t_i$

\item $(\omega_i t\;\;[x_1.u_1,x_2.u_2]) \triangleright  u_i[x_i:=t]$

\end{itemize}

\item The permutative cuts: They appear when the elimination of the 
disjunction
is followed by the elimination rule of a connective.The corresponding 
rule is:

\begin{itemize}

\item $((t\;\;[x_1.u_1,x_2.u_2])\;\;\varepsilon) \triangleright
(t\;\;[x_1.(u_1\; \varepsilon),x_2.(u_2\;\varepsilon)])$

\end{itemize}

\item The classical cuts: They appear when the classical rule is 
followed by the elimination rule of a connective. The corresponding rule is:

\begin{itemize}

\item $(\m a.t\;\; \varepsilon) \triangleright \m
  a.t[a:=^*\varepsilon]$,  where $ t[a:=^*\varepsilon]$ is obtained 
from $t$ by replacing
inductively each subterm in the form $(a \; v)$ by $(a \; (v \; 
\varepsilon))$. 
\end{itemize}
\end{enumerate}
\end{enumerate}

\end{definition}

\begin{notation}
 Let $t$ and $t'$ be $\mathcal{E}$-terms. The notation $t
 \triangleright t'$ means that $t$ reduces to $t'$ by using one step
 of the reduction rules given above. Similarly, $t \triangleright^*
 t'$ means that $t$ reduces to $t'$ by using some steps of the
 reduction rules given above.
\end{notation}

The following result is straightforward.

\begin{theorem}
If $\G \vdash t : A ; \Delta$ and $t \triangleright^* t'$ then $\G 
\vdash t' : A ; \Delta$.
\end{theorem}

We have also the confluence property (see \cite{and}, \cite{deG2} and
\cite{nour}).

\begin{theorem} If $t\triangleright^* t_1$ and $t\triangleright^* t_2$, 
then there exists $t_3$ such that $t_1\triangleright^* t_3$ and 
$t_2\triangleright^* t_3$.

\end{theorem}
\begin{definition}
An $\mathcal{E}$-term $t$ is said to be strongly normalizable if there
is no infinite sequence $(t_i)_{i<\omega}$ of $\mathcal{E}$-terms such
that $t_0 =t$ and $t_i \triangleright t_{i+1}$ for all $i<\omega$. 
\end{definition}

The aim of this paper is to prove the following theorem.

\begin{theorem}\label{SN}
Every typed term is strongly normalizable.
\end{theorem}

In the rest of the paper we consider only typed terms.

\section{Reducibility candidates}

\begin{lemma}\label{sub}
Let $t,u$ and $u'$ be $\mathcal{E}$-terms such that $u\triangleright 
u'$, then:
 \begin{enumerate}
   \item $u[x:=t]\triangleright u'[x:=t]$ and $u[a:=^*t] \triangleright 
u'[a:=^*t]$.
   \item $t[x:=u] \triangleright^* t[x:=u']$ and $t[a:=^*u] 
\triangleright^* t[a:=^*u']$. 
 \end{enumerate}
\end{lemma}

\begin{proof}  
1) By induction on $u$. 2) By induction on $t$.
\end{proof} 

\begin{notation}
The set of strongly normalizable terms (resp.  $\mathcal{E}$-terms) is
denoted by $\mathcal{N}$ (resp.  $\mathcal{N'}$). If $t\in
\mathcal{N'}$, we denoted by $\eta(t)$ the maximal length of the
reduction sequences of $t$.We denote also $\mathcal{N'}^{< \omega}$ the 
set of finite sequences of $\mathcal{N'}$.
\end{notation}

\begin{definition}
Let $\bar{w}=w_1 ... w_n \in \mathcal{N'}^{< \omega}  $, we
say that $\bar{w}$ is a nice sequence iff $w_n$ is the only
$\mathcal{E}$-term in $\bar{w}$ which can be in the
form $[x.u,y.v]$.
\end{definition}

\begin{remark} The intuition behind the notion of the nice sequences
will be given in the proof of the lemma \ref{int}.
\end{remark}

\begin{lemma}\label{nice}
Let $\bar{w}=w_1 ... w_n$ be a nice sequence and $\bar{w'}=w_1
...w_i'... w_n$ where $w_i \triangleright w'_i$. Then $\bar{w'}$ is
also a nice sequence.
\end{lemma}

\begin{proof} 
 This comes from the fact that if $\varepsilon \triangleright 
[x.u,y.v]$
then $\varepsilon = [x.p,y.q]$, where $p \triangleright u$ or
$q\triangleright v$.
\end{proof}

\begin{notation}
\begin{enumerate}
\item The empty sequence is denoted by $\emptyset$. 
\item Let $\bar{w}=w_1 ... w_n$ a sequence of $\mathcal{E}$-terms and
  $t$ a term. Then $(t\;\,\bar{w})$ is $t$ if $n=0$ and $((t \; w_1)
  \; w_2 ... w_n)$ if $n\neq 0$. The term $t[a:=^*\bar{w}]$ is
  obtained from $t$ by replacing inductively each subterm in the form
  $(a \; v)$ by $(a \; (v \;\bar{w}))$.
\item If $\bar{w}=w_1 ... w_n$ is a nice sequence, we denote $\eta
  (\bar{w})=\sum_{i=1}^n\eta(w_i)$.
\end{enumerate}
\end{notation}

\begin{lemma}\label{int}
 Let $\bar{w}$ be a nice sequence. 
     \begin{enumerate}
     \item $(x \, \bar{w})\in \mathcal{N}$.
     \item If $u \in \mathcal{N}$ and $(t[x:=u] \; \bar{w})\in
     \mathcal{N}$, then $((\lambda x.t \; u)\;\bar{w})\in \mathcal{N}$.
     \item If $t_1,t_2 \in \mathcal{N}$ and $(t_i \; \bar{w})\in
     \mathcal{N}$, then $((\<t_1,t_2\> \; \pi_i)\;\bar{w})\in 
\mathcal{N}$. 
     \item If $t,u_1,u_2 \in \mathcal{N}$ and $u_i[x_i:=t] \in 
\mathcal{N}$, then $(\omega_i t\;\;[x_1.u_1,x_2.u_2])\in \mathcal{N}$.
     \item If $t[a:=^*\bar{w}] \in \mathcal{N}$, then $(\mu 
a.t\;\bar{w})\in \mathcal{N}$. 
 \end{enumerate}
\end{lemma}

\begin{proof} 
\begin{enumerate}
\item Let $\bar{w}=w_1...w_n$. All reduction over $(x\;\bar{w})$ take 
place in some $w_i$, because $\bar{w}$ is a nice sequence, and therefore 
the $w_i$ cannot interacte between them via commutative reductions. 
Since all $w_i$ are strongly normalizable, then $(x\;\bar{w})$ itself is 
strongly normalizable. 
\item It suffices to prove that: If $((\lambda x.t
 \;u)\;\bar{w})\triangleright s$, then $s\in \mathcal{N}$. We process
 by induction on $\eta(u) +\eta(t[x:=u]\;\bar{w})$.  Since
 $\bar{w}=w_1...w_n$ is a nice sequence, the $w_i$ cannot interact
 between them via commutative reductions.  We have four possibilities
 for the term $s$.
\begin{itemize}
\item $s = ((\lambda x.t' \;u)\;\bar{w})$ where $t \triangleright t'$: 
By lemma
  \ref{sub}, $(t'[x:=u] \; \bar{w}) \in \mathcal{N}$ and $\eta(u)
 +\eta((t'[x:=u]\;\bar{w})) <  \eta(u)+\eta((t[x:=u]\;\bar{w}))$, then, 
by
  induction hypothesis, $s \in \mathcal{N}$.
\item $s = ((\lambda x.t \;u')\;\bar{w})$ where $u \triangleright u'$: 
By lemma
  \ref{sub}, $(t[x:=u'] \; \bar{w})\in \mathcal{N}$ and $\eta(u')
 +\eta((t[x:=u']\;\bar{w})) <  \eta(u)+\eta((t[x:=u]\;\bar{w}))$, then, 
by
  induction hypothesis,  $s \in \mathcal{N}$.
\item $s = ((\lambda x.t \;u)\;\bar{w'})$ where 
$\bar{w'}=w_1...w'_i...w_n$
 and $w_i \triangleright w'_i$: By lemma \ref{nice}, 
 $\bar{w'}$ is a nice sequence. We have $(t[x:=u] \; \bar{w'})\in
 \mathcal{N}$ and $\eta(u) +\eta((t[x:=u]\;\bar{w'})) <
 \eta(u)+\eta((t[x:=u]\;\bar{w}))$, then, by induction hypothesis, $s
 \in \mathcal{N}$.
\item $s = (t[x:=u] \; \bar{w})$: By hypothesis, $s \in \mathcal{N}$.
\end{itemize}
\item Same proof as 2).
\item Same proof as 2).
\item It suffices also to prove that: If $(\mu 
a.t\;\bar{w})\triangleright
   s$, then $s \in\mathcal{N}$. We process by induction on the pair
   $(lg(\bar{w}),\eta (t[a:=^*\bar{w}]) + \eta(\bar{w}))$ where 
$lg(\bar{w})$ is the
   number of the $\mathcal{E}$-terms in the sequence $\bar{w}$. We have 
three possibilities
 for the term $s$.
\begin{itemize}
\item $s = (\mu a.t'\;\bar{w})$ where $t\triangleright t'$: By lemma
  \ref{sub}, $t'[a:=^*\bar{w}]\in\mathcal{N}$ and
  $\eta(t'[a:=^*\bar{w}]) < \eta(t[a:=^*\bar{w}])$, then, by induction
  hypothesis, $s \in \mathcal{N}$.
\item $s = (\mu a.t\;\bar{w}')$ where $\bar{w'}=w_1...w'_i...w_n$ and 
$w_i
 \triangleright w'_i$: by lemma \ref{nice}, $\bar{w'}$ is a nice
 sequence and, by lemma \ref{sub}, $t[a:=^*\bar{w}']\in\mathcal{N}$
 and $\eta(t[a:=^*\bar{w}']) + \eta(\bar{w}') < \eta(t[a:=^*\bar{w}])
 + \eta(\bar{w})$, then, by induction hypothesis, $s \in \mathcal{N}$.
\item $s = (\mu a.t[a:=^* w_1]\,\bar{w}')$ where
  $\bar{w'}=w_2...w_n$: It is obvious that $\bar{w'}$ is a nice
  sequence and $lg(\bar{w'}) < lg(\bar{w})$. We have $t[a:=^* 
w_1][a:=^*\bar{w'}]=t[a:=^* \bar{w}] \in
  \mathcal{N}$, then, by induction hypothesis, $s \in
  \mathcal{N}$.
\end{itemize}
\end{enumerate} 
\end{proof} 

\begin{lemma}\label{delta}
Let $\bar{w}$ be a nice sequence. 

If $(t \; [x.(u\;\bar{w}),y.(v\;\bar{w})])\in \mathcal{N}$,
then $((t \; [x.u,y.v])\;\bar{w})\in \mathcal{N}$.
\end{lemma}
 
\begin{proof} 
 This is proved by that, from an infinite sequence of reduction 
starting
from $((t \; [x.u,y.v])\;\bar{w})$, an infinite sequence of reduction 
starting
from $(t \; [x.(u\;\bar{w}),y.(v\;\bar{w})])$ can be constructed. A 
complete proof of this result is given in \cite{dav2} in
order to characterize the strongly normalizable terms. 
\end{proof} 

\begin{definition}
\begin{enumerate}
\item We define three functional constructions ($\to , \wedge$ and 
$\vee$)
on subsets of terms:
\begin{enumerate}
\item $K \to L = \{t \in \mathcal{T}/$ for each $u\in K$, $(t \; u) \in 
L\}$.
\item $K \wedge L = \{t \in \mathcal{T}/$ $(t\;\pi_1) \in K$ and 
$(t\;\pi_2) \in L \}$.
\item $K \vee L = \{t \in \mathcal{T}/$ for each $u,v \in \mathcal{N}$:
  If (for each $r \in K$,$s \in L$: $u[x:=r]\in \mathcal{N}$ and $
  v[y:=s]\in \mathcal{N})$, then $(t\;[x.u,y.v])\in \mathcal{N}\}$. 
\end{enumerate}
\item The set $\mathcal{R}$ of the reductibility candidates is the 
smallest
set of subsets of terms containing $ \mathcal{N}$ and closed by the 
functional
constructions $\to , \wedge $ and $\vee$.
\item Let $\bar{w}=w_1 ... w_n$ be a sequence of $\mathcal{E}$-terms, 
we
say that $\bar{w}$ is a good sequence iff for each $1\leq i \leq n$,
$ w_i$ is not in the form $[x.u,y.v]$.
\end{enumerate}
\end{definition}

\begin{lemma}\label{var}
If $R \in \mathcal{R}$, then:
\begin{enumerate}
\item $R \subseteq \mathcal{N}$.
\item $R$ contains the $\l$-variables.
\end{enumerate}
\end{lemma}

\begin{proof} 
We prove, by simultaneous induction, that $R \subseteq
\mathcal{N}$ and for each $\l$-variable $x$ and for each good sequence 
$\bar{w} \in
\mathcal{N}'^{<{\omega}}$, $(x\;\;\bar{w}) \in R$.
\begin{itemize}
  \item $R=\mathcal{N}$: trivial.
  \item $R=R_1\to R_2$: Let $t\in R$. By induction hypothesis, we have 
$x \in R_1$, then $(t\; x)\in R_2$,
therefore, by induction hypothesis, $(t\; x)\in \mathcal{N}$ hence 
$t\in \mathcal{N}$.

Let $\bar{w}\in \mathcal{N}'^{<{\omega}}$ be a good sequence and $v\in
  R_1$. Since $\bar{w}v$ is a good sequence, then, by induction
  hypothesis $(x\;\bar{w}v) \in R_2$, therefore $(x\;\bar{w}) \in R_1 
\to
  R_2$.
 \item $R=R_1\wedge R_2$: Let $t\in R$, then $(t \;\pi_i)\in R_i$ and,
 by induction hypothesis, $(t \;\pi_i)\in \mathcal{N}$, therefore $t\in
 \mathcal{N}$.

Let $\bar{w} \in \mathcal{N}'^{<{\omega}}$ be a good sequence,
then $\bar{w}\pi_i$ is also a good sequence and, by induction
hypothesis, $(x\;\bar{w}\pi_i) \in R_i$, therefore $(x\;\bar{w}) \in 
R$.

  \item $R=R_1\vee R_2$: Let $t\in R$ and $y,z$ two $\l$-variables. By 
induction
  hypothesis, we have, for each $u\in R_1\subseteq
  \mathcal{N}$ and $v\in R_2 \subseteq \mathcal{N}$,
  $y[y:=u] = u\in \mathcal{N}$ and $z[z:=v]= v\in \mathcal{N}$, then
  $(t\;[y.y,z.z]) \in \mathcal{N}$, therefore $t\in \mathcal{N}$.

 Let $\bar{w}\in \mathcal{N}'^{<{\omega}}$ be a good sequence and $u,v
 \in \mathcal{N}$ such that for each $r\in R_1, s\in R_2, u[x:=r]\in
 \mathcal{N}$ and $v[y:=s]\in \mathcal{N}$. We have $[x.u,y.v]\in
 \mathcal{N'}$ because $u$ and $v$ $\in \mathcal{N}$. Thus
 $\bar{w}\;[x.u,y.v]$ is a nice sequence, and by lemma \ref{int},
 $(x\;\bar{w}\;[x.u,y.v]) \in \mathcal{N}$, therefore $(x\;\bar{w})
 \in R$.
\end{itemize}
\end{proof}

\begin{notation} For $S\subseteq {\mathcal{N}'}^{<{\omega}}$, we define 
$S \to
 K=\{t\in \mathcal{T}/$ for each $\bar{w} \in S, (t \;\bar{w}) \in 
K\}$.
\end{notation}

\begin{definition}
A set $X\subseteq {\mathcal{N}'^{<\omega}}$ is said to be nice iff for
each $\bar{w}\in X$, $\bar{w}$ is a nice sequence.
\end{definition}

\begin{lemma}
Let $R \in \mathcal{R}$, then there exists a nice set $X$ such that $R= 
X \to \mathcal{N}$.
\end{lemma}

\begin{proof} By induction on $R$.  
\begin{itemize} 
\item $R=\mathcal{N}$: Take $X=\{\emptyset\}$, it is clear that
$\mathcal{N}=\{\emptyset\} \to\mathcal{N}$.
 \item $R=R_1 \to R_2$: We
have $R_2=X_2 \to \mathcal{N}$ for a nice set $X_2$. Take $X=\{u\;
\bar{v}$ / $u \in R_1, \bar{v}\in X_2\}$. We have $u\; \bar{v}$ is a 
nice
sequence for all $u \in R_1$ and $\bar{v}\in X_2$. Then $X$ is a nice
set and we can easly check that $R = X \to \mathcal{N}$.

    \item $R=R_1 \wedge R_2$: Similar to the previous case. 
    \item $R=R_1 \vee R_2$: Take $X=\{[x.u,y.v]$ / for each $r\in R_1$
and $s\in R_2\;,\; u[x:=r] \in \mathcal{N}$ and $ v[y:=s] \in
\mathcal{N}\}$. We have $X$ is a nice set and, by definition, $R = X
\to \mathcal{N}$.
\end{itemize}
\end{proof}

\begin{remark}
Let $R \in \mathcal{R}$ and $X$ a nice set such that  $R= X \to
\mathcal{N}$. We can suppose that $\emptyset \in X$. Indeed, since $R
\subseteq \mathcal{N}$, we have also $R=X \cup \{ \emptyset \} \to 
\mathcal{N}$.
\end{remark}

\begin{definition}
Let $R \in \mathcal{R}$, we define $R^\perp= \cup \{X \; / \; R= X \to
\mathcal{N}$ and $X$ is a nice set $\}$.
\end{definition}

\begin{lemma}
Let $R \in \mathcal{R}$, then:
\begin{enumerate}
\item $R^\perp$ is a nice set.
\item $R= R^\perp \to \mathcal{N}$.
\end{enumerate} 
\end{lemma} 

\begin{proof}
\begin{enumerate}
\item By definition.
\item This comes also from the fact that: If, for every $i \in I$,
$R=X_i \to \mathcal{N}$, then $R=\cup_{i \in I}X_i \to \mathcal{N}$.
\end{enumerate}
\end{proof}

\begin{remark}  For $R \in \mathcal{R}$, $R^\perp$ is simply the
greatest nice $X$ such that $R= X \to \mathcal{N}$. In fact any nice $X$ such
that $\emptyset \in X$ and $R= X \to \mathcal{N}$ would work as well
as $R^\perp$.
\end{remark}

\begin{lemma}\label{red}
Let $R \in \mathcal{R}$, $t \in R$ and $t \triangleright^* t'$. Then 
$t' \in R$ 
\end{lemma} 

\begin{proof}
Let  $\bar{u} \in R^\perp$. We have $(t \; \bar{u}) \triangleright^*
(t' \; \bar{u})$ and $(t \; \bar{u}) \in \mathcal{N}$, then $(t' \;
\bar{u}) \in \mathcal{N}$. We deduce that $t' \in R^\perp \to
\mathcal{N} = R$.
\end{proof}

\begin{remark}\label{rem} Let $R \in \mathcal{R}$, we have not in 
general
$\mathcal{N}\subseteq R$, but we can prove, by induction, that
$ \m a\mathcal{N} = \{\m a. t$ / $t \in \mathcal{N}$ and $a$ is not
free in $t\} \subseteq R$.
\end{remark}

\section{Proof of the theorem \ref{SN}}

\begin{definition}
An interpretation is a function $I$ from the propositional variables
 to $\mathcal{R}$, which we extend to any formula as follows:
 $I(\perp)=\mathcal{N}$, $I(A\to B)= I(A)\to I(B)$, $I(A\wedge B)=
 I(A)\wedge I(B)$ and $I(A\vee B)= I(A)\vee I(B)$.
\end{definition}
 
\begin{lemma}[Adequation lemma]\label{cle}
Let $\Gamma =\{x_i : A_i\}_{1\le i\le n}$ , $\Delta =\{a_j
: B_j\}_{1\le j\le m}$, $I$ an interpretation, $u_i \in
I(A_i)$, $\bar{v_j} \in I(B_j)^\perp$ and $t$ such that $\Gamma \vdash
t:A \,\,\, ; \, \Delta$. 

Then
$t[x_1:=u_1,...,x_n:=u_n,a_1:=^*\bar{v_1},...,a_m:=^*\bar{v_m}]\in
I(A)$.
\end{lemma}

\begin{proof} 
For each term $s$, we denote

$s[x_1:=u_1,...,x_n:=u_n,a_1:=^*\bar{v_1},...,a_m:=^*\bar{v_m}]$ by
$s'$.

We look at the last used rule in the derivation of $\Gamma \vdash t:A
\,\,\, ; \, \Delta$.
\begin{itemize}
\item ax, $\to_e$ and $\wedge^j_e$: Easy. 
\item $\to_i$: In this case $t=\l x.t_1$ with $\Gamma, x:C \vdash
  t_1:D\,\,\, ; \, \Delta$ and $ A=C \to D$. Let $u\in I(C)$ and
  $\bar{w}\in I(D)^\perp$. By induction hypothesis, we have 
$t'_1[x:=u]\in
  I(D)$, then $ (t'_1[x:=u] \; \bar{w})\in \mathcal{N}$, and,
  by lemma \ref{int}
  $((\l x.t'_1 \, u)\, \bar{w}) \in \mathcal{N}$. Therefore $(\l x.t'_1 
\,\, u)
  \in I(D)$, hence $\l x.t'_1 \in I(C) \to I(D) = I(A)$.
\item $\wedge_i$ and $\vee^j_i$: Similar to $\to_i$.
\item $\vee_e$: In this case
 $t=(t_1\;[x.u,y.v])$ with $\Gamma \vdash t_1:B\vee C \,\,\, ; \,
 \Delta$ , $\Gamma , x:B \vdash u:A\,\,\, ; \, \Delta$ and $\Gamma ,
 y:C \vdash v:A\,\,\, ; \, \Delta$. Let $r\in I(B)$ and $s\in
 I(C)$. By induction hypothesis, we have $t_1'\in I(B)\vee I(C)$,
 $u'[x:=r] \in I(A)$ and $v'[y:=s] \in I(A)$. Let $\bar{w}\in
 I(A)^\perp$, then $(u'[x:=r]\,\bar{w})\in \mathcal{N}$ and
 $(v'[y:=s]\,\bar{w}) \in \mathcal{N}$, therefore $( t_1'\,
 [x.(u'\bar{w}) ,y.(v'\bar{w})])\in \mathcal{N}$. By lemma
 \ref{delta}, $((t_1'\; [x.u' ,y.v'])\bar{w})\in \mathcal{N}$,
 therefore $(t'_1\; [x.u' ,y.v'])\in I(A)$.

\item $abs_e $: In this case $t=\mu a.u $
 and $\Gamma \vdash \m a.u : A\,\,\, ; \, \Delta$. Let $\bar{v} \in
 I(A)^\perp$. It suffies to prove that $((\mu a.u')\; \bar{v}) \in
 \mathcal{N} $. By induction hypothesis, $u'[a:=^*\bar{v}] \in I(\perp)=
 \mathcal{N}$, then, by lemma \ref{int}, $(\m a.u' \; \bar{v})\in
 \mathcal{N}$. Finally $(\mu a. u)' \in I(A)$.

\item $abs_i$: In this case $t= (a_j \;u)$ and $\Gamma \vdash  (a_j\;u) 
:
\perp\,\, ; \, \Delta', a_j:B_j$. We have to prove that $t' \in
\mathcal{N}$, by induction hypothesis, $u' \in I(B_j)$,  then $(u'\;
\bar{v_j}) \in \mathcal{N}$, therefore $t'=(a\;(u'\; \bar{v_j})) \in 
\mathcal{N}$.
\end{itemize}
\end{proof}

\begin{notation}
We denote $I_{\mathcal{N}}$ the interpretation such that, for each
propositional variable $X$, $I_{\mathcal{N}}(X) = \mathcal{N}$.\\
\end{notation}

\begin{proof}[of theorem \ref{SN}]: If $x_1:A_1,...,x_n:A_n \vdash
t:A;a_1:B_1,...,a_m:B_m$, then, by the lemma \ref{var}, $x_i \in
I_{\mathcal{N}}(A_i)$, and, by definition, $\emptyset \in
I_{\mathcal{N}}(B_j)^\perp$. Therefore by lemma \ref{cle},
$t=t[x_1:=x_1,...,x_n:=x_n,a_1:=^*\emptyset,...,a_m:=^*\emptyset] \in
I_{\mathcal{N}}(A)$ and finally, by lemma \ref{var}, $t \in
\mathcal{N}$. 
\end{proof}

\begin{remark} We can give now another proof of remark \ref{rem}:
``if $R \in \mathcal{R}$, the $\m a.\mathcal{N}\subseteq R$''. Let $t=
\l z.\m a.z$, we have $\vdash t : \perp \to p$ for every propositional
variable $p$. By lemma \ref{cle}, for every $R \in \mathcal{R}$, $t
\in
\mathcal{N} \to R$, then, for every $u \in \mathcal{N}$, $(t\;u) \in
R$, therefore, by lemma \ref{red}, $\m a.u \in R$.
\end{remark}


\begin{thebibliography}{99}

\bibitem{and} Y. Andou. {\em Church-Rosser property of simple
reduction for full first-order classical natural deduction.} Annals of
Pure and Applied logic 119 (2003) 225-237.

\bibitem{davnou} R. David and K. Nour. {\em A short proof of the
strong normalization of the simply typed $\l \m$}-calculus. Schedae
Informaticae vol.12, pp. 27-33, 2003.

\bibitem{dav2} R. David and K. Nour. {\em A short proof of the Strong
Normalization of Classical Natural Deduction with Disjunction}. Journal 
of symbolic Logic, vol 68, num 4, pp 1277-1288, 2003.

\bibitem{Gir} J.-Y. Girard, Y. Lafont, P. Taylor.  Proofs and types.
{\it Cambridge University Press\/}, 1986.

\bibitem{deG2} P. de Groote. {\em Strong normalization of classical
natural deduction with disjunction.} In 5th International Conference
on typed lambda calculi and applications, TLCA'01. LNCS (2044),
pp. 182-196.  Springer Verlag, 2001.

\bibitem{JoMa} F. Joachimski and R. Matthes. {\em Short proofs of
normalization for the simply-typed lambda-calculus, permutative
conversions and G\"odel's T}. Archive for Mathematical Logic 42,
pp 59-87 (2003).

\bibitem{matthes} R. Matthes {\em Non-strictly positive fixed-points
for classical natural deduction}. Manuscript, 2003.

\bibitem{nour} K. Nour and  K. Saber {\em Church-Russer property of
full propositional classical natural deduction}. Manuscript, 2004.

\bibitem{Par1} M. Parigot {\em $\lambda \mu$-calculus: An algorithm
interpretation of classical natural deduction.}
 Lecture Notes in Artificial Intelligence (624),
pp. 190-201. Springer Verlag 1992.


\bibitem{Par2} M. Parigot. {\em Proofs of strong normalization for
second order classical natural deduction.} Journal of Symbolic Logic,
62 (4), pp.  1461-1479, 1997.

\end{thebibliography}
\end{document}